\documentclass[
11pt,%
tightenlines,%
twoside,%
onecolumn,%
nofloats,%
nobibnotes,%
nofootinbib,%
superscriptaddress,%
noshowpacs,%
centertags]%
{article}
\usepackage[english]{babel}
\usepackage[latin1]{inputenc}
\usepackage{amssymb}
\usepackage{amsmath}
\usepackage{amsthm}
\usepackage{enumerate}
\usepackage{color}

\newtheorem{theorem}{Theorem}[section]
\newtheorem{lemma}[theorem]{Lemma}
\newtheorem{corollary}[theorem]{Corollary}

\theoremstyle{definition}

\theoremstyle{remark}
\newtheorem{remark}[theorem]{Remark}

\def\N{{\mathbb N}}
\def\R{{\mathbb R}}
\def\Z{{\mathbb Z}}

\begin{document}

\begin{center}{CALCULATIONS OF THE NORMS FOR MONOTONE OPERATORS ON THE CONES OF FUNCTIONS WITH MONOTONICITY PROPERTIES}
\end{center}

\begin{center}
E.~G.~Bakhtigareeva \\
e-mail:~bakhtigareeva-eg@rudn.ru
\end{center}
\begin{center}
M.~L.~Goldman \\
e-mail:~{seulydia@yandex.ru}
\end{center}

\begin{center}
 Peoples Friendship University of Russia (RUDN University),\\ 6 Miklukho-Maklaya St, Moscow, 117198, Russian Federation
 \end{center} 
\begin{center}

Received:~January 18, 2021
\end{center}

\begin{abstract}
The paper is devoted to the problem of exact calculation of the norms in ideal spaces for monotone operators on the cones of functions with monotonicity properties. We implement a general approach to this problem that covers many concrete variants of monotone operators in ideal spaces and different monotonicity conditions for functions. As applications, we calculate the norms of some integral operators on the cones, associate norms over some cones in Lebesgue spaces, the norms of the dilation operator and embedding operators on weighted Lorentz spaces with general weights. Under some more general conditions, we present order sharp estimates for Hardy-type operators on the cones.
\end{abstract}

\section*{INTRODUCTION. MOTIVATION AND ACTUALITY}

The estimates for monotone operators on cones of nonnegative functions play an important role in different branches of analysis, such as the theory of function spaces, the approximation theory, and applications to the theory of partial differential equations. 

Let us describe shortly the general setting of the problem. Let $\mu$ be a nonnegative continuous Borel measure on $\R_+=(0,\infty); L_0=L_0(\R_+)$ be the set of all Borel-measurable real valued functions, $L_0^+=\left\lbrace f \in L_0 : f\geq 0 \right\rbrace;$
$$\dot{L}_0=\left\lbrace f \in L_0: |f|< \infty \quad \mu-\text{almost everywhere} \right\rbrace, \quad  \dot{L}_0^+=\left\lbrace f \in \dot{L}_0: f\geq 0 \right\rbrace.$$
Let $\Omega \subset \dot{L}_0^+$  be a cone of functions, so that $g \in \Omega, \alpha\geq 0 \Rightarrow \alpha g \in \Omega$. Let $A:\dot{L}_0^+ \rightarrow L_0^+ $  be a monotone operator; $H_{\Omega}(A)$ be the norm of restriction of operator $A$ on $\Omega$  
\begin{equation}\label{eq2.3.1}
H_{\Omega}(A)=\sup_{g \in \Omega}\left[ \left(\int\limits_{0}^{\infty} (Ag)^{q}d\gamma \right)^{1/q} \left(\int\limits_{0}^{\infty} g^{p}d\beta \right)^{-1/p} \right], 
\end{equation}
with $0<p,q\leq\infty;\, \beta, \gamma$ being nonnegative continuous Borel measures on $\R_+.$

Typical examples of cones are the following: 
\begin{equation}\label{eq2.3.2}
\Omega=\Omega_{(k)}= \left\lbrace g: 0\leq g(t)<\infty;\, g(t)t^{-k}\downarrow on \quad  \R_{+}\right\rbrace ,  k \in \R;
\end{equation}
\begin{equation}\label{eq2.3.3}
\Omega=\Omega^{(m)}= \left\lbrace g: 0\leq g(t)<\infty;\, g(t)t^{-m}\uparrow on \quad \R_{+}\right\rbrace ,  m \in \R;
\end{equation}

\begin{equation}\label{eq2.3.4}
\Omega=\Omega_{(k)}^{(m)}=\Omega_{k}\cap \Omega^{m}, m<k;
\end{equation}

 We see that $\Omega_{(0)}$ is the cone of nonnegative decreasing functions, $\Omega^{(0)}$  is the cone of nonnegative increasing functions, $\Omega_{(1)}^{(0)}$ is the cone of quasi-concave functions.

The typical operators connected with this problem are the following: the identity operator $A=I$ ; Hardy type operators  $A=H, A=\tilde{H} ;$ convolutions with nonnegative kernels and so on. For example, the corresponding Hardy- and Copson-type  operators are
\begin{equation}\label{eq2.3.5}
(Hg)(t)=\int\limits_{(0, t]} gd\mu; \quad (\tilde{H}g)(t)=\int\limits_{[t, \infty)}gd\mu.
\end{equation}

\textit{The general problem is to find necessary and sufficient conditions for the finiteness: $H_{\Omega}(A)<\infty$, and to establish order-sharp estimates for  $H_{\Omega}(A)$.}

Note that this problem is important for applications. We illustrate it by some examples from different branches of analysis.

\subsection*{Integration theory}\label{subsec2.1.1}
Let $\ddot{L}_0=\ddot{L}_0(\R^n)$  be the subspace of all functions $f: \R^n \rightarrow \R$  measurable with respect to the $n$-dimensional Lebesgue measure $\mu_n$ , they are finite  $\mu_n$-almost everywhere, and such that for $f \in \ddot{L}_0$  the distribution function $\lambda_f$ is not identical to infinity, where 
\begin{equation}\label{eq2.3.6}
\lambda_f(y)=\mu_n\left\lbrace x \in \R^n: |f(x)|>y\right\rbrace, y \in \R_+.
\end{equation}
Then $\lambda_f \in \Omega_{(0)}, \lambda_f(y)\rightarrow 0 \,(y\rightarrow +\infty).$ Let $f^{\ast} \in \Omega_{(0)}$ 
be the \textit{decreasing rearrangement} of the function $f,$ i.e. $f^{\ast}$  is a left-continuous inverse function for the positive decreasing function $\lambda_f,$ namely

\begin{equation}\label{eq2.3.7}
f^{\ast}(t)=\inf \left\lbrace y \in \R_+: \lambda_f(y)< t\right\rbrace, t \in \R_+.
\end{equation}
For  $f^{\ast} \in L_1(0,t), t \in \R_+,$ we define the \textit{elementary maximal function} $f^{\ast\ast}$: 
\begin{equation}
f^{\ast\ast}(t)=t^{-1}\int\limits_{0}^{t}f^{\ast}d\tau.
\label{eq2.3.8}
\end{equation}
It is easy to see that  $f^{\ast\ast}\downarrow, tf^{\ast\ast}\uparrow$ , so that $f^{\ast\ast} \in \Omega_{(0)}^{(-1)}.$ These definitions, and the relations between $f^{\ast}$ and $f^{\ast\ast}$  are considered in details in the books \cite{BL}, \cite{BS}, and \cite{KPS}. 

It is well known that functions $f$ and $f^{\ast}$  are equimeasurable: they have equal distribution functions. 
 Therefore, integral properties of $f$ are determined by   $f^{\ast} \in \Omega_{(0)}$. For example,

$$\int_{\R^n} |f|^p d\mu_n=\int_{\R_+}(f^{\ast})^pd\mu_1, p>0.$$

Let  $Mf$ be the Hardy- Littlewood maximal function for  $f \in  L_1^{loc}(\R^n)$:
\begin{equation}\label{eq2.3.9}
(Mf)(x) =\sup \left\lbrace \mu_n(Q)^{-1}\int_{Q} |f| d\mu_n : x \in Q\right\rbrace, x \in \R^n.
\end{equation}

Here, the supremum is taken over all cubes $Q$  with sides parallel to the axes containing the given point $x$. It is well known that integral properties of the maximal function play important role in problems of functional series theory, Fourier analysis, approximation theory and so on (see \cite{BS}). Integral properties of $Mf$ are determined by $ f^{\ast\ast} \in \Omega_{(0)}^{(-1)}$. For example,

$$\int_{\R^n} |Mf|^pd\mu_n \cong\int_{\R_+}(f^{\ast\ast})^pd\mu_1, p>0.$$

The cause is related to the following fundamental two-sided estimate:  $(Mf)^{*} \cong f^{**}$, see \\
 \cite[Ch. 2]{BS}. Therefore, when we use decreasing rearrangements in study of integral properties of functions (maximal functions), the problems are reduced to the corresponding integral properties on some cones of monotone functions.

\subsection*{Embedding theory of function spaces}\label{subsec2.1.2}
Here we consider some examples appearing in the embedding theory of function spaces.

\begin{center}
	\textbf{Weighted Lorentz spaces with general weights} 
\end{center}

Let us recall two main variants of weighted Lorentz spaces with general weights $0<v,w \in \dot{L}_0^+$:
\begin{equation}\label{eq2.3.10}
\Lambda_{p, v}=\left\lbrace f: ||f||_\Lambda=\left( \int\limits_{0}^{\infty} (f^{\ast})^p vdt\right)^{1/p}<\infty \right\rbrace, 
\end{equation}

\begin{equation}\label{eq2.3.11}
\Gamma_{q, w}=\left\lbrace f: ||f||_\Gamma=\left( \int\limits_{0}^{\infty} (f^{\ast\ast})^q wdt\right)^{1/q}<\infty \right\rbrace. 
\end{equation}

Here $0<p,q <\infty.$  Classical Lorentz spaces correspond to power weights; about general properties and recent developments in the theory of these spaces see \cite{BL}, \cite{BS}, \cite{CPSS}, \cite{CRS}, \cite{FiR}-\cite{GPS}, \cite{LiK}.

Obviously, $\Gamma_{p, v}\subset \Lambda_{p, v}$ because $f^{\ast} \leq f^{\ast\ast}$. From the above definitions, it follows easily that 

\begin{equation}\label{eq2.3.12}
\Lambda_{p, v}\subset \Lambda_{q, w} \Leftrightarrow G_{\Omega_{(0)}}(p,q)<\infty, 0<p,q <\infty,
\end{equation}

where for the cone $\Omega$ we define
\begin{equation}
 G_{\Omega}(p,q)=\sup_{g \in \Omega} \left[ \left( \int\limits_{0}^{\infty}g^qwdt\right)^{1/q}\left( \int\limits_{0}^{\infty}g^pvdt\right)^{-1/p}\right].
 \label{eq2.3.13}
\end{equation}

By the same reasons 

\begin{equation}\label{eq2.3.14}
\Gamma_{p, v}\subset \Gamma_{q, w} \Leftrightarrow G_{\Omega_{(0)}^{(-1)}}(p,q)<\infty, 0<p,q <\infty.
\end{equation}

Further, (\ref{eq2.3.8}), (\ref{eq2.3.10}), and (\ref{eq2.3.11}) imply
\begin{equation}\label{eq2.3.15}
\Lambda_{p, v}\subset \Gamma_{q, w} \Leftrightarrow H_{\Omega_{(0)}}(p,q)<\infty, 0<p,q <\infty,
\end{equation}

where
\begin{equation}\label{eq2.3.16}
H_{\Omega_{(0)}}(p,q)=\sup_{g \in \Omega_{(0)}} \left[ \left( \int\limits_{0}^{\infty}\left(\int\limits_{0}^{t}gd\tau\right)^q t^{-q}wdt\right)^{1/q}\left( \int\limits_{0}^{\infty}g^pvdt\right)^{-1/p}\right].
	\end{equation}
	
	Therefore, the integral properties of the identity operator, and Hardy type operators on the cone of decreasing functions  play a significant role in embedding problems for Lorentz spaces.
	
	\begin{center}
		\textbf{Besov spaces with generalized smoothness}
	\end{center}

Let $f \in L_p(\R^n), 1\leq p\leq \infty.$ Consider the modulus of continuity of order $k$ in  $L_p(\R^n),$ 

\begin{equation}\label{eq2.3.17}
\omega_p^k(f;t)=\sup \left\lbrace ||\Delta_h^k f||_{L_p} : h \in \R^n,\, |h|\leq t\right\rbrace, t \in \R_+.
	\end{equation}
	
Introduce Besov spaces with generalized smoothness:

\begin{equation}\label{eq2.3.18}
B_{p\theta}^{v(\cdot)}(\R^n)=\left\lbrace f \in L_p: ||f||_B=||f||_{L_p}+||f||_b<\infty\right\rbrace,
\end{equation}	

\begin{equation}\label{eq2.3.19}
||f||_b=\left( \int\limits_{0}^{\infty}\omega_p^k(f;t)^{\theta}v(t)dt\right)^{{1}/{\theta}},
\end{equation}
where $0<\theta<\infty.$ Classical Besov space  $B_{p\theta}^{\alpha}(\R^n)$ corresponds to the power weight $v(t)=t^{-\alpha\theta-1}$, $0<\alpha<k,$ see S.M. Nikolskii \cite[Ch. 4]{Nik}. Different variants of Besov spaces with generalized smoothness were considered in the papers of M.Z. Berkolaiko, A.S. Dzhafarov, A. Gogatishvili, M.L. Goldman, D. Haroske, G.A. Kalyabin, H.-G. Leopold, P.I. Lizorkin, S. Moura, Yu.V. Netrusov, B. Opic, P. Oswald, W. Sickel, and many others ; see surveys and references in \cite{FL},  \cite{G1}, \cite{KL}.

Let us note that the embedding problems for Besov spaces may be reduced to the estimates of some positive operators on the cone  $\Omega_{(k)}^{(0)}$. Indeed, it is well known that  

$$\Omega_{(k)}^{(0)}\approx\Omega\equiv \left\lbrace h(t)=\omega_p^k(f;t); f \in L_p(\R^n)\right\rbrace.$$

Here the equivalence means that functions from these cones are pointwise comparable, so that there exists a constant  $c=c(k) \in [1,\infty)$:

$$h_1 \in \Omega \Rightarrow \exists h_2 \in \Omega_{(k)}^{(0)};\quad  h_1  \in \Omega_{(k)}^{(0)} \Rightarrow \exists h_2 \in \Omega: c^{-1}\leq h_1 h_2^{-1}\leq c.$$

Therefore, we can replace $\Omega$  by $\Omega_{(k)}^{(0)}$ in the considerations below. 
Let us present some examples. For $p=q$

 \begin{equation}
 B_{p\theta}^{v(\cdot)}(\R^n) \subset B_{q\tau}^{w(\cdot)}(\R^n) \Leftrightarrow G_{\Omega_{(k)}^{(0)}} (\theta,\tau)<\infty,
 \label{eq2.3.20}
 \end{equation}	
 

Let $1\leq p < q\leq \infty,\, k>n(\frac{1}{p}-\frac{1}{q}).$ Then, 

 \begin{equation}\label{eq2.1.21}
B_{p\theta}^{v(\cdot)}(\R^n) \subset L_q(\R^n) \Leftrightarrow \tilde{G}_{\Omega_{(k)}^{(0)}}<\infty, 
\end{equation}

where 	

 \begin{equation}\label{eq2.1.22}
\tilde{G}_{\Omega_{(k)}^{(0)}}=\sup_{g \in \Omega_{(k)}^{(0)}} \left[ \left( \int\limits_{0}^{\infty}\left[ g(t)t^{-\alpha}\right] ^{q^{\ast}} \frac{dt}{t}\right)^{{1}/{q^{\ast}}}\left(\int\limits_{0}^{\infty}g(t)^{\theta}vdt\right)^{-{1}/{\theta}}\right], 
\end{equation}

\begin{equation}\label{eq2.1.23}
q^{\ast}=q,\, 1\leq q<\infty; \quad q^{\ast}=1,\, q=\infty;\, \alpha= n\left( \frac{1}{p}-\frac{1}{q}\right). 
	\end{equation}
	
	Last conclusion is based on the following sharp estimate (see \cite{G1}, \cite{He}, \cite{U1}, and \cite{U2}):
	for  $1\leq p<q\leq\infty, c=c\left( p, q, k, n\right) \in \R_+,$
	
	\begin{equation}\label{eq2.1.24}
	||f||_{L_q}\leq c\left\lbrace \left( \int\limits_{0}^{\infty}\left[ \omega^k_p(f;t)t^{-\alpha}\right] ^{q^{\ast}} \frac{dt}{t}\right)^{{1}/{q^{\ast}}}+||f||_{L_p}\right\rbrace.
		\end{equation}
		
	\begin{remark} 	The embedding problems (\ref{eq2.3.20}) for $1\leq p, q\leq \infty$, (\ref{eq2.1.21}) were studied in \cite{G1}-\cite{G6}, where the criteria for embeddings in explicit form were established. They stimulated our interest to the estimates for monotone operators on different cones of functions with monotonicity conditions. 
\end{remark}

	\subsection*{Interpolation Theory}
	
	Let  $\left\lbrace A_0, A_1 \right\rbrace $ be a so-called Banach pare (see \cite{KPS}), so that we can add the elements from $A_0$ and $A_1.$ For  $a \in A_0 + A_1$ the famous Peetre's  $K$-functional is defined as
$$K(t, a)\equiv K(t, a; A_0, A_1)=\inf_{a_0+a_1=a} \left\lbrace ||a_0||_{A_0}+t||a_1||_{A_1}\right\rbrace.$$	

Here the infimum is taken over all representations $a_0+a_1=a, a_0 \in A_0, a_1 \in A_1.$

It is obvious that $K(\cdot , a) \in \Omega_{(1)}^{(0)}.$ For 
$v \in L_0^+$ we define the intermediate space 

$$(A_0, A_1)_{v, q}=\left\lbrace a \in A_0+A_1: ||a||_{v, q}=\left( \int_{0}^{\infty}K(t, a)^qv(t)dt\right)^{1/q} <\infty\right\rbrace,$$



These spaces play a significant role in real method of interpolation, see books \cite{BL}, \cite{BS}, \cite{KPS}, \cite{Tri} and references there. The classical case considered by J. Peetre corresponds to $v(t)=t^{-\theta q-1}, 0<\theta <1.$ Many authors (see some names below) studied general spaces. 

When we study the relations between different intermediate spaces, we have to invoke the results on the estimates on the cone  $\Omega_{(1)}^{(0)}.$

The question arises about relationship between the estimates that are valid on the cones with properties of monotonicity and, for example, the estimates on the set of all functions from $\dot{L}_0^+.$ In Section 2 we demonstrate  that such difference can be essential.

These actual problems were actively studied recently by many researchers. In this context, we mention here some names of researchers in alphabetical order:

M. Arino, S. Astashkin, C. and G. Bennett, E. Berezhnoi, J. Bergh, O. Besov, S. Bloom, N. Bokayev, Y. Brudnyi, V. Burenkov, M. Carro, A. Cianchi,  A. Gogatishvili, M. Goldman, K.-G. Grosse-Erdmann, D. Haroske, H. Heinig, S. Janson, G. Kalyabin, A. Kaminska, G. Karadzhov, R. Kerman, V. Kolyada, N. Krugljak, A. Kufner, L. Maligranda, J. Martin,   M. Milman, B. Muckenhoupt, Yu. Netrusov, J. Neves, E. Nursultanov, V. Ovchinnikov, R. Oinarov, K. Oskolkov, B. Opic, L.-E. Persson, L. Pick, D. Prokhorov, E. Sawyer, R. Sharpley, V. Sickel, G. Sinnamon, J. Soria, V. Stepanov,  S. Tikhonov, H. Triebel, P. Ulyanov,...

The paper is organized as follows. In Section 1 we present the survey of general results concerning exact calculations of the norms for monotone operators on the cones of functions with monotonicity conditions in general ideal spaces. These results modify the approach developed in paper \cite{BG} by V. Burenkov and M. Goldman. Section 2 contains applications of these approaches for calculations of the norms for different concrete operators such as integral operators on the cones, the dilation operator and embedding operators in weighted Lorentz spaces with general weights. In Section 3 we present (without proves) the order-sharp estimates for Hardy-type operators on the cones in wider set of parameters than in Section 2.

\section{CALCULATIONS OF THE NORMS OF MONOTONE OPERATOR ON THE CONES}\label{sec1}

\subsection{Some general theorems on the cones of functions with monotonicity properties}\label{sbsec1.1}

Let $(M, \varSigma_{M}, \beta),\, (N, \varSigma_{N}, \gamma)$  be measure spaces with nonnegative $\sigma$-additive full measures $\beta, \gamma;\,$ $S(M, \varSigma_{M}, \beta),\, S(N, \varSigma_{N}, \gamma)$ be the spaces of real-valued measurable functions. First, we recall the concept of the ideal space (shortly: IS) $X \subset S(M, \varSigma_{M}, \beta)$ with the  (quasi)norm $\|\cdot\|_X$ (see \cite{BaG}).

The space $X$ is called an IS if the following conditions are satisfied:

\begin{align*}
(B1)\quad \|f\|_X=0 & \Leftrightarrow f=0 \quad (\beta -a.e.), \quad  \|\alpha f\|_X= \alpha \|f\|_X, \, \alpha \geq 0; \\
& \exists C \in [1, \infty): \|f+g\|_X \leq C\left(\|f\|_X+ \|g\|_X\right) ;\\
(B2)\quad 0\leq f\leq g & \quad (\beta -a.e.)\Rightarrow \|f\|_X \leq \|g\|_X;\\
(B3)\quad 0\leq f_m \uparrow f & \quad (\beta -a.e.)\Rightarrow \|f_m\|_X \uparrow \|f\|_X;\\
(B4)\quad \|f\|_X<\infty & \Rightarrow   |f|< \infty \quad (\beta -a.e.).
\end{align*}

The space $X$ is a normed Banach space if $C=1$ in the triangle inequality,  and it is a quasi-normed quasi-Banach space if $C>1$.

We say that the (quasi)norm  in the ideal space  $X \subset S(M, \varSigma_{M}, \beta)$   is \textit{order-continuous} if 

\begin{equation}\label{eq1.1.1}
\left\lbrace x_{m} \in X, m \in \N;\, 0\leq x_{m}\downarrow 0 \quad \beta-\text{a.e.}\right\rbrace \Rightarrow \|x_{m}\|_{X}\downarrow 0.
\end{equation}

An ideal space $X \subset S(M, \varSigma_{M}, \beta)$ is $l_{p}$-concave for $p \in \R_{+}$ if

\begin{equation}\label{eq1.1.2}
\left( \sum_{m} \|x_{m}\|_{X}^{p}\right)^{1/p} \leq \left\| \left( \sum_{m}|x_{m}|^{p}\right)^{1/p}\right\| _{X}.
\end{equation}

An ideal space $Y \subset S(N, \varSigma_{N}, \gamma)$ is $l_{q}$-convex for $q \in \R_{+}$ if
\begin{equation}\label{eq1.1.3}
\left\|  \left( \sum_{m}|y_{m}|^{q}\right)^{1/q}\right\|_{Y}\leq \left( \sum\limits_{m}||y_{m}||_{Y}^{q}\right)^{1/q}.
\end{equation}

It means that the convergence of the series in the right-hand side of (\ref{eq1.1.2}) or (\ref{eq1.1.3}) implies the convergence of the series in the left-hand side, and the corresponding inequalities hold.

Note that any normed ideal space is $l_{1}$-convex, and $l_{q}$-convexity for $0<q<1$  implies the triangle inequality in the form

\begin{equation}\label{eq1.1.4}
\|f+g\|_{Y}\leq \left( \|f\|_{Y}^{q}+ \|g\|_{Y}^{q}\right)^{1/q} \leq 2^{1/q -1}\left( \|f\|_{Y}+\|g\|_{Y}\right).
\end{equation}

Also, note that  $Y=L_{q}(N, \gamma), \, 0<q<\infty$, is $l_{\rho}$-convex for any $\rho \in (0, q]$  (see Lemma \ref{lm1.3.7} below), and it is  $l_{p}$-concave for any $p \in [q, \infty).$

Let $D \subset X$ be a cone of nonnegative functions such that 
$$f, g \in D;\, c_1, c_2 \geq 0 \Rightarrow c_1 f+c_2 g \in D.$$

An operator $T: D\rightarrow Y$    is called   $l_{r}$-convex for $0<r<\infty$ if for any $f_{m} \in D, m \in \Z$ such that  $\left(\sum\limits_{m}f_{m}^{r}\right)^{1/r} \in D,$

\begin{equation}\label{eq1.1.5}
\left| T\left[  \left( \sum\limits_{m}f_{m}^{r}\right)^{1/r}\right]\right|  \leq \left( \sum\limits_{m}|Tf_{m}|^{r}\right)^{1/r} 
\end{equation}
$\gamma$-almost everywhere on  $N$;  and for all  $f \in D; \alpha\geq 0\Rightarrow T[\alpha f]=\alpha T[f].$

Note that $l_{1}$-convexity of operator $T$ coincides with its countable sublinearity:

$$
\left| T\left[\left(\sum\limits_{m}f_{m}\right)\right]\right|  \leq \left( \sum\limits_{m}|Tf_{m}|\right). 
$$

An operator $T$ is called \textit{ monotone} if
\begin{equation}\label{eq1.1.6}
\left\lbrace f, g \in D;\, 0\leq f \leq g \quad \beta-\text{a.e.}\right\rbrace \Rightarrow \left\lbrace  0\leq Tf \leq Tg \quad \gamma-\text{a.e.}\right\rbrace. 
\end{equation}

An example of   $l_{r}$-convex monotone operator gives the operator 

$$
T\left[f\right] = \left( L[f^{r}]\right)^{1/r}, 
$$

where  $L$ is a countable sublinear monotone operator. Moreover, this formula gives the correspondence between  $l_{r}$- convex and countable sublinear operators.

We will consider the case   $M=J:=(a,b),\, -\infty \leq a , b \leq \infty,$ with nonnegative continuous Borel measure  $\beta,$ and restrictions of operators on the following cones of nonnegative decreasing left-continuous functions on  $J:=(a,b):$

\begin{align}\label{eq1.1.7}
\Omega &= \left\lbrace g \in X:\, 0\leq g \downarrow ;\, g(t)=g(t-0),\, t \in (a,b)\right\rbrace,\nonumber\\ 
\dot{\Omega}&= \left\lbrace g \in \Omega:\, \lim\limits_{t \rightarrow b-0} g(t)=0\right\rbrace.
\end{align}

So, we define the norms of such restrictions 

\begin{equation}\label{eq1.1.8}
\|T\|_{\Omega} =\sup \left\lbrace \|T[g]\|_{Y}:  g \in \Omega,\, \|g\|_{X}\leq 1\right\rbrace,
\end{equation}

\begin{equation}\label{eq1.1.9}
\|T\|_{\dot{\Omega}} =\sup \left\lbrace \|T[g]\|_{Y}:  g \in \dot{\Omega},\, \|g\|_{X}\leq 1\right\rbrace.
\end{equation}

Denote 

\begin{equation}\label{eq1.1.10}
\dot{\Omega}_{0}:=\left\lbrace \chi_{(a, t]}: a<t<b\right\rbrace,\,  \Omega_{0}:=\dot{\Omega}_{0}\bigcup \chi_{(a, b)}; 
\end{equation} 

\begin{equation}\label{eq1.1.11}
F(x, t) =T\left[\chi_{(a, t]}\right](x),\, a<t<b;\quad  F(x, b) =T\left[\chi_{(a, b)}\right] (x).
\end{equation} 

\begin{theorem}\label{th1.1.1}
	Let $0<p \leq q\leq r<\infty;\, X \subset S(J, \beta)$ be an ideal    $l_{p}$-concave space with order-continuous (quasi)norm; 
	\begin{equation}\label{eq1.1.12}
	||\chi_{(a, b)}||_{X}=\infty;
	\end{equation} 
$Y \subset S(N, \gamma)$ be an ideal $l_{q}$-convex space , and $T: \dot{\Omega}\rightarrow Y$ be an $l_{r}$- convex positive operator. Then,

\begin{equation} \label{eq1.1.13}
||T||_{\dot{\Omega}} =||T||_{\dot{\Omega}_{0}}:=\sup\limits_{a<t<b} \left[ ||F(\cdot, t)||_{Y} ||\chi_{(a,t]}(\cdot)||_{X}^{-1}\right].
\end{equation}
\end{theorem}

\begin{remark}\label{rem1.1.2}
Note that in the case of nondegeneracy 
$$||\chi_{(a, b)}||_{X}=\infty \Rightarrow \Omega=\dot{\Omega} \Rightarrow ||T||_{\Omega}=||T||_{\dot{\Omega}},$$
because 

\begin{equation}\label{eq1.1.14}
\left\lbrace 0\leq g\downarrow,\, \lim\limits_{t\rightarrow b-0} g(t)>0\right\rbrace \Rightarrow g \notin X.
\end{equation}

\end{remark}

It means that $||T||_{\Omega}=||T||_{\dot{\Omega}}$ in case (\ref{eq1.1.12}).	
	
Now, let us consider the case 
\begin{equation}\label{eq1.1.15}
||\chi_{(a, b)}||_{X}<\infty.
\end{equation}
 
 Then $\Omega_{0}\subset \Omega$  and we have the following statement.
 
\begin{theorem}\label{th1.1.3}
	Let $0<p\leq q\leq r<\infty;$ $X \subset S(J, \beta)$   be an ideal $l_{p}$- concave space with order-continuous (quasi)norm, and condition (\ref{eq1.1.15}) be satisfied. Let $Y \subset S(J, \gamma)$  be an ideal $l_{q}$ - convex  space, and  $T:\Omega \rightarrow Y$ be an $l_{r}$- convex positive operator. Then,
	
	\begin{equation}\label{eq1.1.16}
	||T||_{\dot{\Omega}} =||T||_{\dot{\Omega}_{0}}:=\sup\limits_{a<t<b} \left[ ||F(\cdot, t)||_{Y} ||\chi_{(a,t]}(\cdot)||_{X}^{-1}\right],
		\end{equation}
		
	\begin{equation}\label{eq1.1.17}
||T||_{\Omega}=||T||_{\Omega_{0}}:=\max\left\lbrace ||T||_{\dot{\Omega}_{0}},  ||F(\cdot, b)||_{Y} ||\chi_{(a,b)}(\cdot)||_{X}^{-1}\right\rbrace.
			\end{equation}
\end{theorem}
 
\subsection{Corollaries}

\begin{corollary}\label{cl1.3.5}
Let $0<p \leq \min \left\lbrace q,r\right\rbrace <\infty;\, X \subset S(J, \beta)$   be an ideal  $l_p$-concave space with order-continuous (quasi)norm, and condition (\ref{eq1.1.12}) be satisfied. Let $Y=L_{q}(N, \gamma),$   and $T$  be an $l_r$-convex positive operator. Then  equality   (\ref{eq1.1.13}) holds.
\end{corollary}

\begin{corollary}\label{cl1.3.6}
	Let $0<p \leq \min\left\lbrace q,r\right\rbrace <\infty;\, X \subset S(J, \beta)$   be an ideal  $l_p$-concave space with order-continuous (quasi)norm, and condition (\ref{eq1.1.15}) be satisfie. Let $Y=L_{q}(N, \gamma),$   and $T$  be an $l_r$-convex positive operator. Then equalities   (\ref{eq1.1.16}), (\ref{eq1.1.17}) hold.
\end{corollary}
 
 To prove these corollaries we need a lemma concerning properties of convexity for Lebesgue spaces.
 
\begin{lemma}\label{lm1.3.7}
	Let  $0<q<\infty.$ Then $Y=L_{q}(N, \gamma)$    is an ideal  $l_{\rho}$-convex  space for any $\rho \in (0,q].$
\end{lemma}

\begin{proof}[Proof of Corollaries \ref{cl1.3.5} and \ref{cl1.3.6}]  
Denote  $\rho=\min\left\lbrace q, r\right\rbrace.$    Then we have $p\leq\rho \leq r.$   According to Lemma \ref{lm1.3.7} we see that  $Y=L_{q}(N, \gamma)$    is   $l_{\rho}$-convex, and we can apply Theorem \ref{th1.1.1} with $\rho$  instead of  $q.$   Thus,  equality (\ref{eq1.1.13}) holds for $Y=L_{q}(N, \gamma).$    Proof of Corollary \ref{cl1.3.6} is the same. We only apply Theorem \ref{th1.1.3} instead of Theorem \ref{th1.1.1}.

\end{proof}

\subsection{Generalization of monotonicity condition}\label{sbsec1.3}

Analogous results are valid on the cones of functions with the monotonicity property with respect to a given positive function $k \in C(J).$ Define 
\begin{equation}\label{eq1.4.1}
\Omega_{k}\equiv \Omega(X,k)\!=\left\lbrace g\!\in X\!: g \geq 0,\, g(t)/k(t)\downarrow;\, g(t)=g(t-0),\, t \in (0,b)\right\rbrace ,
\end{equation}
\begin{equation}\label{eq1.4.2}
\dot{\Omega}_{k}\equiv \dot{\Omega}(X,k)=\left\lbrace g \in \Omega_{k}:\,  g(t)/k(t)\rightarrow 0, t \rightarrow b-0\right\rbrace
\end{equation}
(in this notation  $\Omega_{1}=\Omega,\,\dot{\Omega}_{1}=\dot{\Omega},$  see (\ref{eq1.1.7})). Denote

 \begin{equation}\label{eq1.4.3}
\begin{aligned}
& \dot{\Omega}_{k,0}\equiv \dot{\Omega}_{0}(X,k)=\left\lbrace k\chi_{(a,t]}: a<t<b\right\rbrace,\\
&{\Omega}_{k,0}\equiv {\Omega}_{0}(X,k)=\dot{\Omega}_{k,0}\cup \left\lbrace k\chi_{(a,b)}\right\rbrace.
\end{aligned}
\end{equation}
 
\begin{theorem}\label{th1.4.1}
Let the conditions of Theorem \ref{th1.1.1} hold with the replacement of  $\Omega$     by  $\Omega_{k}$    and condition (\ref{eq1.1.12}) being replaced by 
 \begin{equation}\label{eq1.4.4}
|| k\chi_{(a,b)}||_{X}=\infty.
\end{equation}
Then
 \begin{equation}\label{eq1.4.5}
||T||_{\dot{\Omega}_{k}}=||T||_{\dot{\Omega}_{k,0}}:=\sup_{a<t<b}\left[ ||F_{k}(\cdot, t)||_{Y} ||k(\cdot)\chi_{(a,t]}(\cdot)||^{-1}_{X}\right],
\end{equation}
where we denote 
 \begin{equation}\label{eq1.4.6}
F_{k}(x, t)=T[k\chi_{(a,t]}](x),\,a<t<b;\, F_{k}(x, b)=T[k\chi_{(a,b)}](x).
\end{equation}
\end{theorem}

\begin{remark}\label{rm1.4.2}
	Note that  
	\begin{equation}\label{eq1.4.7}
	||k\chi_{(a,b)}||_{X}=\infty\Rightarrow {\Omega}_{k}=\dot{\Omega}_{k}\Rightarrow ||T||_{{\Omega}_{k}}=||T||_{\dot{\Omega}_{k}}
	\end{equation}
	because
	$\left\lbrace 0\leq g \downarrow,\, \lim\limits_{t\rightarrow b-0}[g(t)/k(t)]>0\right\rbrace$ leads to $ g \notin X.$
\end{remark}

Now we consider the case 
\begin{equation}\label{eq1.4.8}
||k\chi_{(a,b)}||_{X}<\infty.
\end{equation}
In this case  ${\Omega}_{k, 0}\subset {\Omega}_{k},$    and we have the following statement.

\begin{theorem}\label{th1.4.3}
	Let the conditions of Theorem \ref{th1.1.3} hold with the replacement of    $\Omega$   by   ${\Omega}_{k}$   and condition (\ref{eq1.1.15}) being replaced by (\ref{eq1.4.8}). Then
	
	\begin{equation}\label{eq1.4.9}
	||T||_{\dot{\Omega}_{k}}=||T||_{\dot{\Omega}_{k,0}}:=\sup_{a<t<b}\left[ \left\|F_{k}(\cdot, t)\right\|_{Y} \left\| k(\cdot)\chi_{(a,t]}(\cdot)\right\|^{-1}_{X}\right],
	\end{equation}
		\begin{equation}\label{eq1.4.10}
	||T||_{{\Omega}_{k}}=||T||_{{\Omega}_{k,0}}:=\max\left\lbrace \left\| T\right\|_{\dot{\Omega}_{k,0}}, \left\|F_{k}(\cdot, b)\right\|_{Y}||k(\cdot)\chi_{(a,b)}(\cdot)||^{-1}_{X}\right\rbrace.
	\end{equation}
\end{theorem}
\begin{proof}[Proof of Theorem \ref{th1.4.1}]
	Formally, this theorem is more general than Theorem \ref{th1.1.1}, but we can easily reduce it to that one.
	 
	We consider the $l_{r}$-convex positive operator $T: \dot{\Omega}(X, k)\rightarrow Y.$ Let us define
	\begin{equation}\label{eq1.4.11}
X_{k}:=\left\lbrace f \in S(J, \beta):\, fk \in X\right\rbrace =\left\lbrace f=g/k:\, g \in X\right\rbrace, ||f||_{X_{k}}= ||kf||_{X}.
\end{equation}	
	Then we have the equivalence:
		\begin{equation}\label{eq1.4.12}
	g \in \dot{\Omega}(X, k)\Leftrightarrow f=g/k \in \dot{\Omega}(X_{k}, 1);\, ||f||_{X_{k}}= ||kf||_{X},
	\end{equation}	
	see (\ref{eq1.4.1}), (\ref{eq1.4.2}), (\ref{eq1.4.11}). Therefore,
	\begin{equation*}
	||T||_{\dot{\Omega}(X,k)}\!=\!\sup\left\lbrace \frac{||T[g]||_{Y}}{||g||_{X}}\!: g \!\in\!\dot{\Omega}(X, k)\right\rbrace\!=\!
	\sup\left\lbrace  \frac{||T[kf]||_{Y}}{||f||_{X_{k}}}\!: 	f\!\in\!\dot{\Omega}(X_{k}, 1)\right\rbrace.
	\end{equation*}
	
	Note that $X_{k}$   together with $X$    is an ideal  $l_{p}$-concave space with order-continuous (quasi)norm, and the operator
	$$T_{k}: \dot{\Omega}(X_{k}, 1)\rightarrow Y; \, T_{k}[f]:=T[kf], \, f \in \dot{\Omega}(X_{k}, 1),$$
	is  $l_{r}$-convex together with the operator $T.$ Thus, we apply Theorem \ref{th1.1.1} and obtain that
		\begin{equation*}
	||T||_{\dot{\Omega}(X,k)}=\sup_{a<t<b}\left\lbrace \left\|T[k\chi_{(a,t]}]\right\|_{Y}\left\|k\chi_{(a,t]}\right\|^{-1}_{X}\right\rbrace.
	\end{equation*}
\end{proof}

\begin{proof}[Proof of Theorem \ref{th1.4.3}]
It is essentially the same as one for Theorem  \ref{th1.4.1}. We only replace $\dot{\Omega}(X,k)$  by ${\Omega}(X,k),$ $\dot{\Omega}(X_k, 1)$  by  ${\Omega}(X_k, 1),$ and apply Theorem \ref{th1.1.3} instead of Theorem \ref{th1.1.1}, taking into account that (\ref{eq1.1.15}) has now the form \eqref{eq1.4.8}.
\end{proof}
	
\section{Applications}\label{sec1.5}

\subsection{Calculation of the norm of an integral operator on the cone of functions with the monotonicity property}\label{subsec1.5.1}

Let $K=K(x, \tau)$ be a nonnegative measurable function of variables $(x, \tau) \in N\otimes J,$ where $(N; \gamma)$  and  $(J; \mu)$ are measure spaces with the nonnegative     $\sigma$-finite    $\sigma$-additive measure $\gamma$  and the nonnegative continuous Borel measure    $\mu$ on  $J=(a,b).$
\begin{equation}\label{eq1.5.1}
T_{r\mu}[f](x)=\left( \int_{(a, b)}K(x, \tau)|f(\tau)|^{r}d\mu(\tau)\right)^{1/r},\quad r \in (0, \infty).
\end{equation}
It is  an $l_{r}$-convex monotone operator. For $r=1$ its restriction on the set of nonnegative  $\mu$ - measurable functions coincides with the restriction of the linear integral operator

\begin{equation}\label{eq1.5.2}
T[f](x)=\int_{(a, b)}K(x, \tau)f(\tau)d\mu(\tau).
\end{equation}

All the results of Section 1 are applicable here. In particular, for the restrictions of  $T_{r\mu}$ on the cone    $\Omega_{k}$   application  of Theorems \ref{th1.4.1} and \ref{th1.4.3} gives the following results.

\begin{theorem}\label{th1.5.1}
Let  $0<p\leq q\leq r<\infty;$ $X \subset S(J, \beta)$ be an ideal $l_p$--concave space with order-continuous (quasi) norm;  $Y \subset S(N, \gamma)$   be an $l_q$-convex ideal space, and
\begin{equation}\label{eq1.5.3}
\|k\chi_{(a,b)}\|_{X}=\infty.
\end{equation}
Then,
\begin{equation}\label{eq1.5.4}
\|T_{r\mu}\|_{\Omega_k}=\|T_{r\mu}\|_{\dot{\Omega}_{k,0}}:=\sup_{a<t<b}\left\lbrace  \|T_{r\mu}[k\chi_{(a,t]}]\|_{Y}\|k\chi_{(a,t]}\|_{X}^{-1}\right\rbrace .
\end{equation}
Here
\begin{equation}\label{eq1.5.5}
T_{r\mu}[k\chi_{(a,t]}](x)= \int_{(a, t]}K(x, \tau)k(\tau)^{r}d\mu(\tau),\quad x \in N.
\end{equation}
In the case
\begin{equation}\label{eq1.5.6}
\|k\chi_{(a,b)}\|_{X}<\infty,
\end{equation}
  we have $\|T_{r\mu}\|_{\dot{\Omega}_{k}}=\|T_{r\mu}\|_{\dot{\Omega}_{k,0}}$ (see \eqref{eq1.5.4}),
  
  \begin{equation}\label{eq1.5.7}
  \|T_{r\mu}\|_{\Omega_k}=\max \left\lbrace \|T_{r\mu}\|_{\dot{\Omega}_{k,0}};\quad  \|T_{r\mu}[k\chi_{(a,b)}]\|_{Y}\|k\chi_{(a,b)}\|_{X}^{-1}\right\rbrace .
  \end{equation}
\end{theorem}

\begin{remark}\label{rm1.5.2}
For the restrictions on the cone $\Omega=\Omega_1$ of nonnegative decreasing left-continuous functions we have to set $k(\tau)=1$ in \eqref{eq1.5.3}-- \eqref{eq1.5.7}.
\end{remark}

\begin{remark}\label{rm1.5.3}
In the case  $Y=L_q(N, \gamma)$ the results of Theorem \ref{th1.5.1} hold if  $0<p\leq \min \left\lbrace q,r\right\rbrace <\infty$, see Corollaries \ref{cl1.3.5} and \ref{cl1.3.6}.
\end{remark}

\begin{remark}\label{rm1.5.4}
	As a particular case of operator \eqref{eq1.5.1} we will consider later the case where  $(N, \gamma)=(J, \gamma)$ with the nonnegative continuous  Borel measure $\gamma$ on $J=(a,b)$, and $T_{r\mu}$ coincides with the generalized Hardy-type operator
	\begin{equation}\label{eq1.5.8}
	A_{r\mu}[f](x)=\left( \int_{(a, x]}|f(\tau)|^{r}d\mu(\tau)\right)^{1/r},\quad x \in (a, b).
	\end{equation}
	Then,
	\begin{align}
	A_{r\mu}[k\chi_{(a,t]}](x)=\left( \int_{(a, x]}k(\tau)^{r}d\mu(\tau)\right)^{1/r},\quad x \leq t;\nonumber\\
	A_{r\mu}[k\chi_{(a,t]}](x)=\left( \int_{(a, t]}k(\tau)^{r}d\mu(\tau)\right)^{1/r},\quad x > t.\label{eq1.5.9}	
	\end{align}
	and we have the equalities: in case \eqref{eq1.5.3}

\begin{equation}\label{eq1.5.10}
\|A_{r\mu}\|_{\Omega_k}=\|A_{r\mu}\|_{\dot{\Omega}_{k,0}}:=\sup_{a<t<b}\left\lbrace \left\|A_{r\mu}[k\chi_{(a,t]}]\right\|_{Y}\left\|k\chi_{(a,t]}\right\|_{X}^{-1}\right\rbrace;
\end{equation}
in case \eqref{eq1.5.6} formula \eqref{eq1.5.10} remains true for  $\|A_{r\mu}\|_{\dot{\Omega}_{k}}$, but

 \begin{equation}\label{eq1.5.11}
\|A_{r\mu}\|_{\Omega_k}=\max \left\lbrace \|A_{r\mu}\|_{\dot{\Omega}_{k,0}};\quad  \left\| A_{r\mu}[k\chi_{(a,b)}]\right\|_{Y}\left\|k\chi_{(a,b)}\right\|_{X}^{-1}\right\rbrace .
\end{equation}

\end{remark}
\subsection{Calculation of the associated norm over the cone of functions with the monotonicity property}\label{subsec1.5.2}

Let us describe the results concerning calculation of associate norms mentioned in {\bf Introduction}. For   $0<p\leq \infty$ and for nonnegative Lebesgue-measurable functions $f, g \in \dot{L}^{+}_{0}(\R_{+})$   let us define
\begin{align*}
C_{p}(f, g)=\left(\, \int_{\R_+}fgdt\right) \left(\int_{\R_+}g^{p}dt\right)^{-1/p};\\
A(f, p)=\sup_{g \in  \dot{L}^{+}_{0}} C_{p}(f, g);\quad B(f, p)=\sup_{g \in  \Omega_k} C_{p}(f, g).
\end{align*}
These quantities present the associate norms on the sets $\dot{L}^{+}_{0}(\R_{+})$ and $\Omega_k$, respectively. Here,  $k$ is a given positive continuous function on $[0,\infty)$. Classical results are the following:
 \begin{equation}\label{eq1.5.12}
A(f, p)=\begin{cases}\infty,& 0<p<1, \, f\neq 0;\\ 
\|f\|_{L_{p'}}, & 1\leq p\leq\infty,\frac{1}{p}+\frac{1}{p'}=1.\end{cases} 
\end{equation}
The first assertion illustrates nonexistence of bounded linear non-zero functionals on $L_p(\R_+)$  for $0<p<1$. The second one shows the calculation of the associate norm for the norm in $L_p(\R_+)$, $1\leq p\leq\infty$.

The results on the cone are essentially different. Namely, 
let $0<p\leq 1$. We define
\begin{equation*}
D_{p}(f, k, t)=\left( \int\limits_{(0,t]}fk d\tau\right) \left(\int\limits_{(0,t]}k^{p}d\tau\right)^{-1/p}.
\end{equation*}
\begin{theorem}\label{th1.5.5}
	In  the notation of this Section
	\begin{equation}\label{eq1.5.13}
	B_k(f, p)=\sup \left\lbrace  D_{p}(f, k, t):\, 0<t<\infty\right\rbrace.
	\end{equation}
\end{theorem}

\begin{proof}
	These results follow immediately from Theorems \ref{th1.4.1} and \ref{th1.4.3}. Indeed, here $X=L_p(\R_{+}, \beta)$, $0<p\leq 1$, $\beta$  is the Lebesgue measure on $\R_{+}$; $Y=L_1(\R_{+}, \gamma)$, $d\gamma(t)=g(t)dt$; $T=I$ is the identity operator which is $l_1$-convex, and  Theorems \ref{th1.4.1} and \ref{th1.4.3} are applicable here.
	
	If $\int\limits_{(0, \infty)}k^pd\tau=\infty$, then \eqref{eq1.5.13} follows from Theorem \ref{th1.4.1}. In the case  $\int\limits_{(0, \infty)}k^pd\tau<\infty$  we apply Theorem \ref{th1.4.3}, so that \eqref{eq1.4.10} implies 
	\begin{equation}\label{eq1.5.14}
	B_k(f, p)=\max \left\lbrace \sup \left\lbrace  D_{p}(f, k, t):\, 0<t<\infty\right\rbrace; D_{p}(f, k, \infty)\right\rbrace.
	\end{equation}
	
	In particular, the condition $\int\limits_{(0, \infty)}fk\,d\tau<\infty$ is necessary for the finiteness of $B_k(f, p)$. Nevertheless, let us note that the right--hand sides of \eqref{eq1.5.13} and \eqref{eq1.5.14} coincide because $D_{p}(f, k, t)$ is  continuous on $\R_+$.	
\end{proof}

\begin{remark}
	The results on the cone $\Omega_1$ and $\dot{\Omega}_1$ of decreasing functions we obtain by setting here $k(\tau)\equiv 1$. 
\end{remark}

\subsection{Calculation of the norm of the dilation operator 	and triangle inequality in Lorentz space}\label{subsec1.5.3}

Let $L_0=L_0(\R^n)$  be the space of all Lebesgue-measurable functions $f: \R^n\rightarrow \R$; $\ddot{L}_0=\ddot{L}_0(\R^n)$  be the subspace of functions having the decreasing rearrangements $f^{\ast}$ which is not identical to infinity (see \eqref{eq2.3.7}). 



Note that $f^{\ast}$ is a nonnegative decreasing left-continuous function on  $\R_+$.  Now, let  $0<p<\infty$, $v$  be a positive Lebesgue--measurable function on  $\R_+$.    
Introduce a weighted Lorentz space:  
\begin{equation}\label{eq1.5.17}
\Lambda_{p, v}=\left\lbrace f \in \ddot{L}_0:\, \|f\|_{\Lambda_{p, v}}=\left( \int_{0}^{\infty}(f^{\ast})^{p}v d\tau\right)^{1/p}<\infty\right\rbrace. 
\end{equation}

For  $v(\tau)\equiv 1$, we have $\Lambda_{p, v}=L_p(\R^n)$. The criterion for nontriviality is
\begin{equation}\label{eq1.5.18}
\Lambda_{p, v}\neq{0} \Leftrightarrow \exists t \in \R_+:\, \int\limits_{(0, t)}vd\tau<\infty. 
\end{equation}

Now we consider the dilation operator $\sigma_m,\, m \in \R_+$, 
\begin{equation}\label{eq1.5.19}
\sigma_m(f)(x)=f(m^{-1}x), x \in \R^n. 
\end{equation}
Then,
\begin{equation}\label{eq1.5.20}
\sigma_m(f)^{\ast}=\sigma_{m^n}(f^{\ast}). 
\end{equation}

Indeed, by \eqref{eq2.3.6}
\begin{equation}\label{eq1.5.21}
\begin{aligned}
\lambda_{\sigma_m(f)}(y)&=\mu_n \left\lbrace x \in \R^n:\, |f(m^{-1}x)|>y\right\rbrace \\
&=m^n\mu_n \left\lbrace x \in \R^n:\, |f(x)|>y\right\rbrace=m^n\lambda_{f}(y), 
\end{aligned}
\end{equation}
so that
\begin{equation*}
\begin{aligned}
\left[ \sigma_m(f)\right] ^{\ast}(\tau)&=\inf \left\lbrace y>0:\, \lambda_{\sigma_m(f)}(y)< \tau\right\rbrace \\
&=\inf \left\lbrace y>0:\, \lambda_{f}(y)< m^{-n}\tau\right\rbrace=f^{\ast}(m^{-n}\tau)=\sigma_{m^n}(f^{\ast})(\tau). 
\end{aligned}
\end{equation*}

We define 
\begin{equation}\label{eq1.5.22}
V(\tau)=\int\limits_{(0, \tau]}v(\rho)d\rho,\, \tau \in \R_+,\quad V(\infty)=\int_{\R_+}v(\rho)d\rho=\lim\limits_{\tau \rightarrow \infty}V(\tau).
\end{equation}

\begin{theorem}\label{th1.5.6}
	In the notation of this Section let condition \eqref{eq1.5.18} be satisfied. For $E=\Lambda_{p, v}$   the following equality for the norm of   $\sigma_{m}: E\rightarrow E$   holds: 
	\begin{equation}\label{eq1.5.23}
	\|\sigma_{m}\|=\sup_{\tau \in \R_{ +}} \left[ V(m^n\tau)V(\tau)^{-1}\right]^{1/p};
	\end{equation}
\end{theorem}

\begin{proof}
	We have 
		\begin{align*}
		&\|\sigma_{m}(f)\|_{\Lambda_{p, v}}=\left( \int_{\R_{+}}\left[\sigma_{m}(f)^{\ast}\right]^{p}vd\tau\right)^{1/p}= \left( \int_{\R_{+}}\left[\sigma_{m^n}(f^{\ast})\right]^{p}vd\tau\right)^{1/p}\\ 
		&= \left( \int_{\R_{+}}\left[f^{\ast}(m^{-n}\tau)\right]^{p}v(\tau)d\tau\right)^{1/p}=m^{n}\left( \int_{\R_{+}}\left[f^{\ast}(\rho)\right]^{p}v(m^{n}\rho)d\rho\right)^{1/p},	
		\end{align*}
		Therefore,
		\begin{align*}
		\|\sigma_{m}\|=m^{n}\sup_{f \in \Lambda_{p, v}}\left[\left(\int_{\R_{+}}\left[f^{\ast}(\rho)\right]^{p}v(m^{n}\rho)d\rho\right)^{1/p} \left( \int_{\R_{+}}\left[f^{\ast}(\rho)\right]^{p}v(\rho)d\rho\right)^{-1/p}\right]. 
		\end{align*}
	Note that in this context 

	\begin{equation*}
	\Omega\equiv \left\lbrace g:\, 0\leq g \downarrow,\, g(t)=g(t-0);\, \int_{\R_{+}}g^pvd\rho<\infty\right\rbrace,
	\end{equation*}
	\begin{equation*}
	\dot{\Omega}\equiv \left\lbrace g \in \Omega:\, g(t)\rightarrow 0\,(t\rightarrow \infty)\right\rbrace.
	\end{equation*}	
	
	\begin{equation}\label{eq1.5.24}
	V(\infty)=\infty,\, f \in \Lambda_{p, v}\Leftrightarrow g=f^{\ast} \in \dot{\Omega};
	\end{equation}
		\begin{equation}\label{eq1.5.25}
	V(\infty)<\infty,\, f \in \Lambda_{p, v}\Leftrightarrow g=f^{\ast} \in \Omega.
	\end{equation}
	Therefore,  $V(\infty)=\infty$ or $V(\infty)<\infty$  imply (respectively)
		\begin{align*}
\|\sigma_{m}\|=m^{n}\sup_{g \in \dot{\Omega}}\left[\left(\int_{\R_{+}}g(\rho)^{p}v(m^{n}\rho)d\rho\right)^{1/p} \left( \int_{\R_{+}}g(\rho)^{p}v(\rho)d\rho\right)^{-1/p}\right]; 
	\end{align*}
	\begin{align*}
	\|\sigma_{m}\|=m^{n}\sup_{g \in {\Omega}}\left[\left(\int_{\R_{+}}g(\rho)^{p}v(m^{n}\rho)d\rho\right)^{1/p} \left( \int_{\R_{+}}g(\rho)^{p}v(\rho)d\rho\right)^{-1/p}\right]; 
	\end{align*}
	
	In both cases we come to formula  \eqref{eq1.5.23}   applying Corollaries \ref{cl1.3.5} and \ref{cl1.3.6}, respectively. In the last case we take into account that
\begin{equation*}
\sup_{0<\tau \leq \infty}	\left[ V(m^n\tau)V(\tau)^{-1}\right]^{1/p}=\sup_{\tau \in \R_{+}}\left[ V(m^n\tau)V(\tau)^{-1}\right]^{1/p},
\end{equation*}
because the function $V$  is positive and continuous on $\R_{+}$.	
\end{proof}

\begin{remark}\label{rm1.5.7}
	Note that
	\begin{equation*}
\|\sigma_{m}\|=\sup_{\tau \in \R_{+}}\left[ V(m^n\tau)V(\tau)^{-1}\right]^{1/p}<\infty \Leftrightarrow V \in \Delta_2.
	\end{equation*}    
	It is the well-known  $\Delta_2$- condition for the function $V$. 
\end{remark}
Let us discuss now the triangle inequality for a Lorentz space $\Lambda_{p,v}$.

\begin{theorem}\label{th1.5.8}
In the notation of this Section, let  $\|\sigma_{2^{1/n}}\|<\infty$. Then\\ $f, g$ $\in \Lambda_{p,v}$ implies $f+g \in \Lambda_{p,v}$, and 
	\begin{equation}\label{eq1.5.26}
\|f+g\|_{\Lambda_{p,v}}\leq \|\sigma_{2^{1/n}}\|\left(  \|f\|_{\Lambda_{p,v}}+\|g\|_{\Lambda_{p,v}}\right),\, 1\leq p<\infty;
\end{equation} 
\begin{align}\label{eq1.5.27}
&\|f+g\|_{\Lambda_{p,v}}\leq \|\sigma_{2^{1/n}}\|\left(  \|f\|_{\Lambda_{p,v}}^{p}+\|g\|_{\Lambda_{p,v}}^{p}\right)^{1/p}\nonumber \\
&\leq 2^{1/p -1}\|\sigma_{2^{1/n}}\|\left(  \|f\|_{\Lambda_{p,v}}+\|g\|_{\Lambda_{p,v}}\right),\, 0< p<1.
\end{align}
\end{theorem}

\begin{proof}
	For $f, g \in \ddot{L}_0$  there is the well-known  inequality (see \cite{BS})
	\begin{equation*}
	(f+g)^{\ast}(\tau_1+\tau_2)\leq f^{\ast}(\tau_1)+g^{\ast}(\tau_2),\, 0<\tau_1, \tau_2 <\infty.
	\end{equation*}
	Thus, 
		\begin{equation*}
	(f+g)^{\ast}(\tau)\leq f^{\ast}(\tau/2)+g^{\ast}(\tau/2)=\sigma_{2}(f^{\ast})(\tau)+\sigma_{2}(g^{\ast})(\tau),\, 0<\tau <\infty,
	\end{equation*}
	so that in view of \eqref{eq1.5.20}
	\begin{equation*}
	(f+g)^{\ast}\leq [\sigma_{2^{1/n}}(f)]^{\ast}+[\sigma_{2^{1/n}}(g)]^{\ast}.
	\end{equation*}
		\begin{equation*}
	\|f+g\|_{\Lambda_{p,v}}\!=\!\left( \int\limits_{0}^{\infty}\![(f+g)^{\ast}]^{p}vd\tau\right)^{1/p}\!\leq\!\left( \int\limits_{0}^{\infty}\! [(\sigma_{2^{1/n}}(f))^{\ast}+(\sigma_{2^{1/n}}(g))^{\ast}]^{p}vd\tau\right)^{1/p}\!.
	\end{equation*}
	Now, we apply the triangle inequality for Lebesgue spaces in the form
		\begin{equation}\label{eq1.5.28}
	\|h_1+h_2\|_{L_{p,v}}\leq \|h_1\|_{L_{p,v}}+\|h_2\|_{L_{p,v}},\, 1\leq p<\infty;
	 \end{equation}
	 \begin{equation}\label{eq1.5.29}
	 \begin{aligned}
	 &\|h_1+h_2\|_{L_{p,v}}\leq \left( \|h_1\|_{L_{p,v}}^{p}+ \|h_2\|_{L_{p,v}}^{p}\right)^{1/p}\\
	&\leq 2^{1/p-1}\left( \|h_1\|_{L_{p,v}}+ \|h_2\|_{L_{p,v}}\right),\, 0<p<1.
	\end{aligned}
	\end{equation}
	Here,  $h_1=(\sigma_{2^{1/n}}(f))^{\ast}$, $h_2=(\sigma_{2^{1/n}}(g))^{\ast}$. Therefore, for $1\leq p<\infty$,
	\begin{equation*}\begin{aligned}
	\|f+g\|_{\Lambda_{p,v}}&\leq \|(\sigma_{2^{1/n}}(f))^{\ast}\|_{L_{p,v}}+ \|(\sigma_{2^{1/n}}(g))^{\ast}\|_{L_{p,v}}\\
	&\leq\|(\sigma_{2^{1/n}}\|\left(\|f\|_{\Lambda_{p,v}}+\|g\|_{\Lambda_{p,v}}\right).
		\end{aligned} 
		\end{equation*}
		Analogously, we obtain \eqref{eq1.5.27}  applying \eqref{eq1.5.29}.
\end{proof}

\begin{remark}\label{rm1.5.9}
Let us recall the equivalence 
\begin{equation}\label{eq1.5.30}
\|\sigma_{2^{1/n}}\|=\sup_{\tau \in \R_{+}} \left[V(2\tau)V(\tau)^{-1} \right]^{1/p}<\infty \Leftrightarrow V \in \Delta_2.
\end{equation}
Thus, this condition guaranties that $\Lambda_{p,v}$   is  a linear (quasi)normed space. It is known that condition \eqref{eq1.5.30} is necessary for such linearity, see \cite{CPSS,CRS}.
\end{remark}

\subsection{Calculation of the norms of embedding operators for Lorentz spaces}

We preserve the notation of Section 2 and consider the conditions for the embedding of Lorentz spaces.
\begin{equation*}
\Lambda_{p,v}\subset  \Lambda_{q,w},\, 0<p\leq q.
\end{equation*}
Here $v,w$ are positive Lebesgue-measurable functions on  $(0,\infty)$. 

\begin{theorem}\label{th1.5.10}
	For $0<p\leq q < \infty$  the following equivalence takes place:
\begin{equation}\label{eq1.5.31}
\Lambda_{p,v}\subset  \Lambda_{q,w} \Leftrightarrow \, A_{pq}:=\sup_{t \in \R_+} \left[ W(t)^{1/q}V(t)^{-1/p}\right] <\infty,
\end{equation}	 
	where $W(t)=\int\limits_{(0, t]}wd\tau$.  Moreover, $\|J\|=A_{pq}$   for the embedding operator $J:\Lambda_{p,v}\rightarrow  \Lambda_{q,w}$.  
\end{theorem}

\begin{proof}
	We have 
	\begin{equation}\label{eq1.5.32}
	\|J\|=\sup_{f \in \Lambda_{p,v}} \left[ \left( \int_{\R_+}(f^{\ast})^{q}wd\tau\right) ^{1/q}\left( \int_{\R_+}(f^{\ast})^{p}vd\tau\right)^{-1/p}\right].
	\end{equation}
	Now, the equivalences \eqref{eq1.5.24} or \eqref{eq1.5.25} admit us to come to the cones $\dot{\Omega}$  or $\Omega$. Note that the identity  operator is $l_r$-convex for any $r \in \R_+$, and for $0<p\leq q <\infty$ we can apply Corollaries \ref{cl1.3.5}  and \ref{cl1.3.6}, respectively. In both cases we have
	\begin{equation*}
	\|J\|=\sup_{t \in \R_+} \left[ W(t)^{1/q}V(t)^{-1/p}\right].
	\end{equation*}
	As before we take into account that 
	\begin{equation*}
\sup_{t \in (0, \infty]} \left[ W(t)^{1/q}V(t)^{-1/p}\right]=\sup_{t \in \R_+} \left[ W(t)^{1/q}V(t)^{-1/p}\right].
	\end{equation*}
	Now let us consider the other variant of Lorentz spaces. For the decreasing rearrangement $f^{\ast}$ we introduce the mean-value $f_{r,\mu}^{\ast\ast}$ with respect to the nonnegative continuous Borel measure $\mu$ on $\R_+$:
	\begin{equation*}
	f_{r,\mu}^{\ast\ast}(\tau)=\left( \frac{1}{M(\tau)}\int\limits_{(0, \tau]}(f^{\ast})^rd\mu\right)^{1/r};\quad M(\tau)=\int\limits_{(0, \tau]}d\mu,\, \tau \in \R_+.
	\end{equation*}
	For $0<q,r<\infty$  the generalized $\Gamma$- Lorentz space is determined as follows
	\begin{equation*}
	\Gamma_{qr}(\gamma,\mu)=\left\lbrace f \in \ddot{L}_{0}:\, \|f\|_{\Gamma_{qr}(\gamma,\mu)}=\left(\int_{0}^{\infty} \left( f_{r,\mu}^{\ast\ast}\right)^{q}d\gamma\right)^{1/q}<\infty\right\rbrace,
	\end{equation*}
	where $\gamma$ is a nonnegative continuous Borel measure.
	Classical variant of this space (see \cite{CPSS, CRS}) we obtain in the case: $r=1,\,\mu$ is  a Lebesgue measure, $d\gamma(\tau)=w(\tau)d\tau$.  Some variants of generalization were introduced in  \cite{FiR}-\cite{GPS}. Here we consider the case $0<p\leq \min\left\lbrace q,r\right\rbrace$. 
\end{proof}
 \begin{theorem}\label{th1.5.11}
 	Let  $0<p\leq \min\left\lbrace q,r\right\rbrace$. Then the following equivalences take place.
 	\begin{enumerate}
 		\item If $V(\infty)=\infty$, then
 		\begin{equation}\label{eq1.5.33}
 	\Lambda_{p,v}\subset \Gamma_{qr}(\gamma,\mu)\Leftrightarrow \dot{J}_{pqr}:= \sup_{t \in \R_{ +}}\left[ W_{qr}(t)^{1/q}V(t)^{-1/p}\right] <\infty,
 		\end{equation}
 		and $\|J\|=\dot{J}_{pqr}$  for the embedding operator $J: \Lambda_{p,v}\rightarrow \Gamma_{qr}(\gamma,\mu)$, where 
 		\begin{equation}\label{eq1.5.34}
 		 W_{qr}(t)=\int_{\R_+}\left( \frac{M(\min\left\lbrace t, \tau\right\rbrace )}{M(\tau)}\right) ^{q/r}d\gamma(\tau).
 		\end{equation}
 		
 		\item If $V(\infty)<\infty$, then
 			\begin{equation}\label{eq1.5.35}
 		\Lambda_{p,v}\subset \Gamma_{qr}(\gamma,\mu)\Leftrightarrow J_{pqr}:= \sup_{t \in (0, \infty]}\left[ W_{qr}(t)^{1/q}V(t)^{-1/p}\right] <\infty,
 		\end{equation}
 		and $\|J\|=J_{pqr}$. 
 	\end{enumerate}
 \end{theorem}

\begin{proof}
	\begin{enumerate}
		\item We define  
			\begin{equation}\label{eq1.5.36}
	X=L_{p,v}(\R_+),\, Y=L_{q}(\R_+, \tilde{\gamma});\, d \tilde{\gamma}(\tau)=M(\tau)^{-q/r}d\gamma(\tau),\, \tau \in \R_+. 
		\end{equation}
		We have by definitions
			\begin{equation*}
		\|f\|_{\Gamma_{qr}(\gamma,\mu)}=\left( \int_{\R_+}\left(f_{r\mu}^{**}\right) ^qd\gamma\right)^{1/q}=
		\left( \int_{\R_+}\left(\int\limits_{(0, \tau]}(f^{*})^{r}d\mu\right)^{q/r}d\tilde{\gamma}(\tau)\right)^{1/q},
		\end{equation*}
		\begin{equation*}
		\|J\|=\sup_{f \in \Lambda_{p,v}}\left[ \left( \int_{\R_+}\left(\int\limits_{(0, \tau]}(f^{*})^{r}d\mu\right)^{q/r}d\tilde{\gamma}(\tau)\right)^{1/q} \left( \int_{\R_+}(f^{*})^{p}vd\tau\right)^{-1/p}\right].
		\end{equation*}
		Now we take into account equivalences \eqref{eq1.5.24} and \eqref{eq1.5.25} and come to the cones $\dot{\Omega}$  or $\Omega$.
		
		\item    If $V(\infty)=\infty$  we have $\Omega=\dot{\Omega}$  and 
		\begin{equation*}
		\|J\|=\sup_{g \in \dot{\Omega}}\left[ \left( \int_{\R_+}\left(\int\limits_{(0, \tau]}g^{r}d\mu\right)^{q/r}d\tilde{\gamma}(\tau)\right)^{1/q} \left( \int_{\R_+}g^{p}vd\tau\right)^{-1/p}\right].
		\end{equation*}
		
		In the case  $g(\tau)=\chi_{(0,t]}(\tau)$   we have 
			\begin{equation*} 
		\int\limits_{(0, \tau]}g^{r}d\mu=M(\min\left\lbrace t,\tau\right\rbrace );\, \int_{\R_+}g^{p}vd\tau=V(t).	
			\end{equation*}     
			Thus, condition \eqref{eq1.1.13} is satisfied, and application of Corollary \ref{cl1.3.5} implies   
			\begin{equation*}
			\|J\|=\sup_{t \in \R_+}\left[ \left( \int_{\R_+}M(\min\left\lbrace t,\tau\right\rbrace )^{q/r}d\tilde{\gamma}(\tau)\right)^{1/q}V(t)^{-1/p}\right].
			\end{equation*}     
			Therefore, equality \eqref{eq1.5.33} holds.
			
			If $V(\infty)<\infty$,  then 
			\begin{equation*}
			\|J\|=\sup_{g \in \Omega}\left[ \left( \int_{\R_+}\left(\int\limits_{(0, \tau]}g^{r}d\mu\right)^{q/r}d\tilde{\gamma}(\tau)\right)^{1/q} \left( \int_{\R_+}g^{p}vd\tau\right)^{-1/p}\right].
			\end{equation*}
			In this case condition \eqref{eq1.1.15} is satisfied, and application of Corollary \ref{cl1.3.6} implies
			\begin{equation*}
			\|J\|=\sup_{t \in (0, \infty]}\left[ \left( \int_{\R_+}M(\min\left\lbrace t,\tau\right\rbrace )^{q/r}d\tilde{\gamma}(\tau)\right)^{1/q}V(t)^{-1/p}\right].
			\end{equation*}     
			Therefore, equality \eqref{eq1.5.35} holds.
	\end{enumerate}
\end{proof}
\begin{remark}\label{rm1.5.12}
	We obtain \eqref{eq1.5.33}  and \eqref{eq1.5.35} applying results for general operators (Corollaries \ref{cl1.3.5} and \ref{cl1.3.6}). Now, let us note that $M(\min\left\lbrace t,\tau\right\rbrace )$ increases by $t$  for any $\tau \in \R_+$,  $\lim\limits_{t \rightarrow +\infty}M(\min\left\lbrace t,\tau\right\rbrace)=M(\tau)$.    Therefore, by Levy's theorem 
	\begin{equation}\label{eq1.5.37}
	W_{qr}(t)\uparrow W_{qr}(\infty)=\int_{\R_{+}}d\gamma.
	\end{equation}
	Therefore, in the case  $0<V(\infty)<\infty$    we have
		\begin{equation}\label{eq1.5.38}
\dot{J}_{pqr}\geq \lim_{t \rightarrow +\infty}\left[ 	W_{qr}(t)^{1/q}V(t)^{-1/p}\right] =\left[  W_{qr}(\infty)V(\infty)^{-1/p}\right].
	\end{equation}
	It means that in the considered case $J_{pqr}=\dot{J}_{pqr}$.
\end{remark}
\begin{corollary}\label{cl1.5.13}
	In Theorem \ref{th1.5.11} let $p=q\leq r,\, d\gamma(\tau)=v(\tau)d\tau.$ Then the equality $\Lambda_{p,v}=\Gamma_{pr}(\gamma,\mu)$ is equivalent to the following  condition:  
	\begin{equation*}
	\|J\|=\sup_{t \in \R_+}\left[ \left( \int_{\R_+}M(\min\left\lbrace t,\tau\right\rbrace)^{p/r}M(\tau)^{-p/r}v(\tau)d\tau
	\right)^{1/p}V(t)^{-1/p}\right]<\infty.
	\end{equation*} 
	Moreover, 
		\begin{equation}\label{eq1.5.39}
\|f\|_{\Lambda_{p,v}}\leq \|f\|_{\Gamma_{pr}(\gamma,\mu)}\leq 
	\|J\|\|f\|_{\Lambda_{p,v}}.
	\end{equation}
\end{corollary}
 Indeed, the right-hand-side inequality follows by Theorem \ref{th1.5.11}. For the left-hand-side one it is enough to note that 
 	\begin{equation*}
 f_{r\mu}^{**}(\tau)=\left(\frac{1}{M(\tau)} \int\limits_{(0, \tau]}(f^{*})^{r}d\mu \right)^{1/r}\geq f^{*}(\tau).
 \end{equation*}
 \begin{remark}\label{rm1.5.14}
 	In this Section we considered  the applications to the theory of Lorentz spaces of the results obtained in  Section 1 for operators on the cones with monotonicity properties. They are connected with some restrictions on parameters such as $0<p\leq q\leq r$ in Theorems \ref{th1.1.1}, \ref{th1.1.3}, \ref{th1.4.1}, \ref{th1.4.3}, or $0<p\leq q$ in Theorem \ref{th1.5.10}, or $0<p\leq \min\left\lbrace q,r\right\rbrace $  in Corollaries \ref{cl1.3.5}, \ref{cl1.3.6}, and  Theorem \ref{th1.5.11}. The technic of consideration connected with such restrictions admit us to calculate \textit{exactly} the norms of related operators. In the next Section we formulate a result  about order-sharp estimates for Hardy-type operators on the general cone $ \Omega_{k}$ without such restrictions on parameters. 
 \end{remark}

\section{ESTIMATES FOR HARDY-TYPE OPERATORS}

Let $b \in (0, \infty], \ p, q, r \in (0, \infty)$ be fixed, $\beta, \gamma, \mu$ be nonnegative continuous Borel measures on  $(0, b); \beta \in N_{p}(k)$ that is 
$$0<\omega_{kp}(t):=\left(\int\limits_{(0,t)}k^qd\beta \right) \in C(0,b);\quad \omega_{kp}(+0)=0,\, \omega_{kp}(b-0)=\infty.$$
 Consider 
	\begin{equation}
\label{eq3.1}
H_{ \Omega_{k}} \equiv H_{ \Omega_{k}}   \left( {A_{r\mu}}  \right)=
\sup\limits_{f\in \Omega_{k }}  \left[\left( \int\limits_{(0, b)} \left( A_{r\mu}f\right)^{q} d \gamma \right)^{ 1 / q}\left(\int\limits_{(0, b)}  {f^{ p} d \beta }  \right)^{- 1 / p}\right].
	\end{equation}
	Here,
	
	\begin{equation}
	\label{eq3.2}
	\left(  {A_{r \mu}   f}  \right) \left(  { t}
	\right)=\left( \int\limits_{( { 0, t}]} { f^{r} d \mu
	}\right)^{{1}/{r}}, \, t \in (0, b),
	\end{equation}
	is the Hardy-type operator. 
	We need some notation to formulate the results. Define 
	\begin{equation}
	\label{eq3.3}
	\Psi_{r}(\tau)\equiv 	\Psi_{r}(k, \mu; \tau)=\left( \int\limits_{( {0, \tau } ]}
	k^{r} d \mu\right)^{{1}/{r}} ,  \quad \tau \in (0, b),
	\end{equation}
		\begin{equation}
	\label{eq3.4}
\mathfrak{J}_{p q r}\equiv 	\mathfrak{J}_{p q r}(k, \beta, \gamma, \mu)=\sup_{t \in (0, b)}\left[\left( \int\limits_{\left( {0, b}  \right)}
	\Psi_{r}(\min \left\lbrace t, \tau \right\rbrace )^{q}d\gamma(\tau)\right)^{{1}/{q}}\omega_{k p}(t)^{-1}\right].  
	\end{equation}
	\begin{equation}
	\label{eq3.5}
	V_{ pr}  \left(  { t}  \right)\equiv 	V_{ pr}  \left(  {k, \beta, \mu; t}  \right)=\mathop {\sup} \limits_{\tau \in
	({0, t }  ]} \left[  \Psi_{r} \left(\tau \right) \omega _{k p} \left(  { \tau }\right)^{-1} \right],\quad p \leq r;
	\end{equation}
	and for $p>r, \sigma=\frac{pr}{p-r},$ 
	\begin{equation}
	\label{eq3.6}
	V_{ pr}  \left(  { t}  \right)=\left\lbrace  {\int\limits_{({0, t }]} {\Psi_{r}  \left({\tau }\right)}^{ \sigma} \left(- d \left[ {\omega _{ k p} \left({ \tau }  \right)}^{ -\sigma}\right]\right)  } +\left[\Psi_{r}(t) \omega _{ k p}(t)^{-1}\right]^{\sigma}  \right\rbrace^{ \frac{1}{\sigma}};
	\end{equation}
	\begin{equation}
	\label{eq3.7}
	W_{ q}  \left(  { t}  \right)=	W_{ q}  \left(  {\gamma; t}  \right)=\left(\ {\int\limits_{\left(
			{ t, b }  \right)} {d \gamma} }  \right)^{ 1  / q}, \, t \in (0, b);
	\end{equation}
	
	\begin{equation}
	\label{eq3.8}
	E_{ p q r} \equiv E_{ p q}\left( k, r, \beta, \gamma, \mu\right) =\mathop {\sup} \limits_{t \in (0, b)}  \left[\left(\int\limits_{ (0, t]} {\Psi_{r}^{ q} d \gamma}	\right)^{ 1 / q}\omega _{k p}  \left( t\right)^{-1}\right],\quad p \leq  q;
	\end{equation}
	and for $p > q, s=\frac{pq}{p-q},$
	\begin{equation}
	\label{eq3.9}
	E_{ p q r}  = \left\lbrace \int\limits_{(0, b) }  \left(\int\limits_{ (0,t]} \Psi_{r}^{ q}d \gamma   \right)^{{ s}\mathord{\left/ {\vphantom {{ s}  { q}}} \right.
					\kern-\nulldelimiterspace} { q}}\left(  {- d \left[
				\omega _{k p}(t)^{-s}   \right]}
			\right)  \right\rbrace ^{ { 1}  \mathord{\left/ {\vphantom {{ 1}  { s}}}
			\right. \kern-\nulldelimiterspace} { s}};
	\end{equation}
	\begin{equation}
	\label{eq3.10}
	F_{ p q r}\equiv 	F_{ p q} (k, r, \beta, \gamma, \mu) =\mathop {\sup} \limits_{t\in 
	(0, b)}  \left[
	{V_{ pr}  \left(  { t}  \right)W_{ q} \left(  { t}  \right)}
	\right],\quad p \leq  q,
	\end{equation}
	\begin{equation}
	\label{eq3.11}
	F_{ p q r} =\left\lbrace {\int\limits_{(0, b) }}  V_{ p r}(t)^{ s}\left( -d \left[W_{ q}(t)^{ s}\right]  \right) \right\rbrace^{ 1 / s},\quad p > q, s=\frac{pq}{p-q}. 
	\end{equation}
	
	\begin{theorem}\label{th3.1} Let the above notation and conditions hold, $\beta \in N_p(k)$. Then the following assertions take place: 
		\begin{equation}
		\label{eq3.12}
		 H_{ \Omega _{ k}}=\mathfrak{J}_{p q r},\, p\leq \min \left\lbrace q; r\right\rbrace;
		\end{equation}
			\begin{equation}
			\label{eq3.13}
			c ^{ - 1}\left(  { (E_{ p q r })^{r} +(F_{ p q r})^{r} }  \right)^{{1}/{r}}\leq H_{ \Omega _{ k}}(A_{r \mu})  \leq c\left( (E_{ p q r})^{r} +(F_{p q r})^{r}\right)^{{1}/{r}}, 
			\end{equation}
		for $p> \min \left\lbrace q; r\right\rbrace$	with $c	=c\left(p, q, r \right)\in \left[ { 1, \infty} \right).$ 
		\end{theorem}
	\begin{remark}\label{rm3.1} Note that the convergence of integral in (\ref{eq3.11}) and integration by part imply the following equality
		\begin{equation}\label{eq3.11'}
		F_{ p q r} =\left\lbrace {\int\limits_{(0, b) }}  W_{ q}(t)^{ s}\left( d \left[V_{ pr}(t)^{ s}\right]  \right) \right\rbrace^{ 1 / s}. 
	\end{equation}
\end{remark}
	\begin{remark} Obviously, (\ref{eq3.13}) implies estimate
		\begin{equation}
	\label{eq3.13'}
	c_{1}^{ - 1}\left(  { E_{ p q r } +F_{ p q r} }  \right)\leq H_{ \Omega _{ k}} \leq c_{1}\left( E_{ p q r} +F_{p q r}\right),
	\end{equation}
	with some constant $c_{1}	=c_{1}\left(p, q, r \right)\in \left[ { 1, \infty} \right).$ 
\end{remark}

\section*{Acknowledgments}
This work is supported by the Ministry of Science and Higher Education of the Russian Federation: agreement no. 075-03-2020-223/3 (FSSF-2020-0018).



\end{document}